\documentclass[a4paper,12pt]{amsart}
\usepackage[T1]{fontenc}
\usepackage[utf8]{inputenc}
\usepackage[left=3.00cm, right=2.00cm, top=3.00cm, bottom=2.00cm]{geometry}
\usepackage{bussproofs}
\usepackage{graphicx}
\usepackage{amsthm}
\usepackage{amsmath}
\usepackage{amsfonts}
\usepackage{amssymb}
\usepackage{amsbsy}
\usepackage{fancyhdr}
\usepackage{mathrsfs}
\usepackage{lmodern}
\usepackage{graphicx}
\usepackage{xcolor}
\usepackage[colorlinks,allcolors=blue]{hyperref}
\usepackage[nameinlink,capitalize]{cleveref}
\usepackage{tikz-cd}
\usepackage{mathtools}
\usepackage{parskip}
\usepackage{enumitem}
\usepackage[english]{babel}

\newcommand{\F}{\mathbb{F}}

\newcommand{\C}{\mathcal{C}}

\newcommand{\X}{\mathcal{X}}

\newcommand{\CL}{\mathcal{C}_\mathcal{L}}

\def\cD{\mathcal{D}}

\DeclareMathOperator{\Supp}{Supp}

\DeclareMathOperator{\ev}{ev}

\everymath{\displaystyle}

\makeatletter
\renewcommand*\env@matrix[1][*\c@MaxMatrixCols c]{%
	\hskip -\arraycolsep
	\let\@ifnextchar\new@ifnextchar
	\array{#1}}
\makeatother

\newtheorem{definition}{Definition}
\newtheorem{example}{Example}
\newtheorem{lemma}{Lemma}
\newtheorem{theorem}{Theorem}
\newtheorem{remark}{Remark}
\newtheorem{proposition}{Proposition}
\newtheorem{corollary}{Corollary}
\newtheorem{conjecture}{Expectation}

\title[Explicit LCP of MDS Codes, LCD Codes on Hyperelliptic Curves via Mumford Pairs]{Explicit LCP of MDS Codes and LCD Codes on Hyperelliptic Curves via Mumford Representation}

\author{Adler Marques, Yuri da Silva, and Saeed Tafazolian
}
\date{}

\address{Universidade Federal do Rio de Janeiro --	Instituto de Matemática,	Cidade Universitária,	CEP 21941-909, Rio de Janeiro, Brazil}
\email{adler@im.ufrj.br}

\address{Departamento de Matem\'{a}tica - Instituto de Matem\'{a}tica, Estat\'{i}stica e Computação Cient\'{i}fica
	(IMECC) - Universidade Estadual de Campinas (UNICAMP), Rua S\'{e}rgio Buarque de Holanda, 651, Cidade Universit\'{a}ria,  Zeferino Vaz, Campinas, SP 13083-859, Brazil}

\email{y225979@dac.unicamp.br}
\email{saeed@unicamp.br}

\begin{document}

	\begin{abstract}
		We study algebraic geometry codes on hyperelliptic curves of genus $g \geq 2$ equipped with a rational Weierstrass point, with a focus on complementarity properties. Our first contribution is a characterization of non-special divisors of degree $g$ and $g-1$ via the polynomial degrees of their reduced Mumford representation, reducing a classical geometric problem to a degree test on univariate polynomials. Using this, we construct Linear Complementary Pairs (LCP) of codes via polynomial arithmetic on the Jacobian and provide a criterion in terms of Mumford degrees for the resulting codes to be Maximum Distance Separable (MDS). Under a $2$-torsion condition in the Jacobian, equivalently a divisibility condition on the Mumford polynomials, we obtain explicit multipliers that turn these pairs into Linear Complementary Dual (LCD) codes. We give an MDS LCP with parameters $[8,4,5]_{16}$ in characteristic two, an MDS LCD code with parameters $[10,5,6]_{25}$ on the maximal hyperelliptic curve $y^2=x^5+x$, and an LCD code with parameters $[28,14]_{49}$ on $y^2=x^7+x$. Motivated by the explicit $q=5$ computation, we conjecture for odd prime powers $q\ge5$ the existence of MDS LCD codes with parameters $[2q,q,q+1]_{q^2}$ arising from this construction.
	\end{abstract}
	
	
	\maketitle
	\section{Introduction}
	\label{sec:introduction}
	
	Linear codes with complementary properties have become fundamental objects in both coding theory and cryptographic applications. Two particularly important classes are Linear Complementary Dual (LCD) codes and Linear Complementary Pairs (LCP) of codes. LCD codes were introduced by Massey~\cite{Massey1992}, who showed that they provide optimum linear coding solutions for the two-user binary adder channel and for nearest-neighbour decoding. More recently, Carlet and Guilley~\cite{CarletGuilley2016} demonstrated that LCD codes are instrumental in securing cryptographic implementations against side-channel attacks and non-invasive fault attacks. The generalization to LCP of codes~\cite{Carlet2018}, where a pair $(\mathcal{C}_1, \mathcal{C}_2)$ of linear codes of length $n$ over a finite field $\mathbb{F}_q$ satisfies $\mathcal{C}_1 \oplus \mathcal{C}_2 = \mathbb{F}_q^n$, is a key component of Direct Sum Masking (DSM). This countermeasure simultaneously protects against Side-Channel Attacks (SCA) and Fault-Injection Attacks (FIA) on cryptographic hardware~\cite{Bringer14,Carlet2018}. The classical LCD condition is recovered as the special case $\mathcal{C}_2 = \mathcal{C}_1^\perp$. For an LCP $(\mathcal{C}_1, \mathcal{C}_2)$, the minimum distance $d(\mathcal{C}_1)$ determines the resistance to FIA, while $d(\mathcal{C}_2^\perp)$ determines the resistance to SCA; the \emph{security parameter} $d(\mathcal{C}_1, \mathcal{C}_2) := \min\{d(\mathcal{C}_1), d(\mathcal{C}_2^\perp)\}$ jointly quantifies the overall security~\cite{Carlet2018}.
	
	The explicit construction of LCD codes and LCP of codes is an active area of research. The problem becomes more intricate when additional structural properties, such as the Maximum Distance Separable (MDS) condition, are required. Achieving the Singleton bound $d = n - k + 1$ is desirable as it maximizes both error-correcting capability and cryptographic security for a given length and dimension. Two primary approaches have emerged in the literature to construct these optimal codes: methods based on classical evaluation codes on rational curves, and methods based on general Algebraic Geometry (AG) codes.
	
	Following the first approach, several works have focused on explicitly constructing LCD MDS codes and LCP of MDS codes from Generalized Reed-Solomon (GRS) codes and their variants, such as Twisted Reed-Solomon (TRS) codes~\cite{Beelen2018, Jin2017, ChenLiu2018, ShiYueYang2018, QianZhang2018, BeelenTRS2017, LiuLiu2021}. Alongside these explicit constructions, foundational results by Carlet et al.~\cite{CarletEquiv2018, CarletMDS2018} established that for $q > 3$, every linear code is equivalent to an LCD code. However, these explicit algebraic constructions operate exclusively within the genus-zero setting. Because the underlying curves are projective lines, the properties of these evaluation codes rely on classical Vandermonde determinants and standard divisor arithmetic, rather than on advanced geometric tools.
	
	To extend these results, it is natural to consider curves of higher genus ($g > 0$) through the framework of AG codes. Introduced by Goppa~\cite{Goppa1981}, these codes are constructed by evaluating functions on arbitrary smooth projective curves over finite fields. As demonstrated by Tsfasman, Vlăduț, and Zink~\cite{TVZ1982}, sequences of AG codes can asymptotically exceed the Gilbert-Varshamov bound. AG codes also provide a natural framework for constructing long codes with complementary properties. Nevertheless, as framed by Bhowmick, Dalai, and Mesnager~\cite{Bhowmick2024}, the explicit construction of LCPs and LCD codes in this higher-genus setting requires the explicit description of \emph{non-special} divisors of degree $g-1$, which remains a challenging problem. The existence of such non-special divisors is guaranteed in general by Ballet and Le Brigand~\cite{Ballet2006} for $q \geq 4$; see \cite{BK2025} for a recent survey on the topic.
	
	Recent advances have primarily relied on Kummer extensions to construct these non-special divisors of degree $g-1$ explicitly \cite{Moreno2024, Castellanos2025, Mendoza2026, 4author, MST2026}. These methods have enabled the construction of LCPs and LCD codes across various function fields, including asymptotic constructions. For instance, in \cite{MQ2025}, the explicit description of non-special divisors within a tower of function fields defined by Artin-Schreier extensions enabled the construction of sequences of LCD codes and LCP of codes that attain the Tsfasman-Vlăduț-Zink (TVZ) bound. Despite this progress, current results often depend on specific structural properties of given extensions or non-constructive existence theorems. A general framework for explicitly constructing non-special divisors for $g>1$ remains an open problem.
	
	In this work, we address this problem for hyperelliptic curves of genus $g \geq 2$ by exploiting the Mumford representation. Every divisor class on a hyperelliptic curve admits a unique reduced representative encoded by a pair of univariate polynomials $(u, v) \in \mathbb{F}_q[x] \times \mathbb{F}_q[x]$, where $u$ is monic and $\deg v < \deg u \leq g$ \cite{Mum84}. We provide a complete algebraic characterization of non-special divisors in terms of this representation, translating a geometric property into explicit polynomial conditions. Specifically, we prove that a divisor $A$ of degree $g-1$ is non-special if and only if $\deg u_A = g$, where $(u_A, v_A)$ is the reduced Mumford representative of the class $[A-(g-1)P_\infty]$. Similarly, a divisor $A$ of degree $g$ is non-special if and only if $\deg u_A \geq g-1$.
	
	The paper is organized as follows. In Section \ref{sec:characterization}, we establish the complete algebraic characterization of non-special divisors of degree $g-1$ and $g$ using the Mumford representation. In Section \ref{sec:code_constructions}, we present an explicit construction of LCP of hyperelliptic codes over $\F_q$ via the Chinese Remainder Theorem and provide a criterion for the codes to be MDS. We also introduce a 2-torsion condition on the divisor class, expressed as $u_A \mid (2v_A + h)$, which allows for the explicit computation of residues of a differential form, yielding a multiplier-based construction of LCD codes from hyperelliptic curves. Finally, we apply our framework to the maximal hyperelliptic curve $y^2 = x^q + x$ over $\mathbb{F}_{q^2}$. By exploiting the maximality of the curve and a Jacobian counting argument, we confirm the existence of Mumford data satisfying both MDS and LCD conditions, providing concrete examples of MDS LCD codes with parameters $[2q, q, q+1]_{q^2}$.

	\section{Preliminaries}
	\label{sec:preliminaries}
	
	In this section we fix notation and recall the background needed in
	the rest of the paper: algebraic function fields and Riemann-Roch
	spaces, algebraic geometry (AG) codes, linear complementary pairs
	(LCP) and linear complementary dual (LCD) codes, and the Mumford
	representation of divisor classes on hyperelliptic curves.
	
	\subsection{Algebraic Function Fields}
	\label{ssec:ff}
	
	Let $\mathbb{F}_q$ be a finite field with $q$ elements and let
	$\overline{\mathbb{F}}_q$ denote its algebraic closure. Let
	$\mathcal{X}$ be a smooth, projective, absolutely irreducible
	algebraic curve defined over $\mathbb{F}_q$, of genus $g$, with
	function field $\mathbb{F}_q(\mathcal{X})$. We denote by
	$\mathrm{Div}(\mathcal{X})$ the divisor group of $\mathcal{X}$, that
	is, the free abelian group generated by the closed points (places)
	of $\mathcal{X}$. For a place $P$ of $\mathcal{X}$, $v_P$ denotes the
	associated discrete valuation and $\deg(P)$ its degree over
	$\mathbb{F}_q$.
	
	A divisor $D=\sum_{P}m_PP\in\mathrm{Div}(\mathcal{X})$, where
	$m_P=0$ for almost all $P$, has support
	\[
	\mathrm{Supp}(D)=\{P\in\mathcal{X}\mid m_P\neq0\}
	\]
	and degree $\deg(D)=\sum_{P}m_P\deg(P)$. The group
	$\mathrm{Div}(\mathcal{X})$ carries the partial order
	\[
	D=\sum_P m_PP\ \le\ G=\sum_P n_PP
	\quad\Longleftrightarrow\quad
	m_P\le n_P \text{ for all } P.
	\]
	For a nonzero $z\in\mathbb{F}_q(\mathcal{X})$ we write
	$\mathrm{div}(z)=\mathrm{div}_0(z)-\mathrm{div}_\infty(z)
	=\sum_P v_P(z)P$
	for its principal, zero and pole divisors, respectively. Two
	divisors $D_1,D_2\in\mathrm{Div}(\mathcal X)$ are linearly
	equivalent, written $D_1\sim D_2$, if $D_1-D_2=\mathrm{div}(z)$ for
	some $z\in\mathbb{F}_q(\mathcal{X})^*$.
	
	Given $G\in\mathrm{Div}(\mathcal{X})$, the \emph{Riemann-Roch
		space} associated with $G$ is the $\mathbb F_q$-vector space
	\[
	\mathcal{L}(G)=\{z\in\mathbb{F}_q(\mathcal{X})\mid
	v_P(z)\ge -v_P(G)\ \text{for all } P\}\cup\{0\},
	\]
	of dimension $\ell(G):=\dim_{\mathbb F_q}\mathcal{L}(G)$; this
	dimension depends only on the linear equivalence class of $G$.
	
	Let $\Omega_{\mathcal X}=\{f\,dx\mid f\in\mathbb F_q(\mathcal X)\}$
	be the space of rational differential forms on $\mathcal X$, for a
	fixed separating element $x$. For a local parameter $t$ at a place
	$P$ and $\omega=f\,dt\in\Omega_{\mathcal X}$, set
	$v_P(\omega):=v_P(f)$, and define the divisor of $\omega$ by
	$\mathrm{div}(\omega)=\sum_P v_P(\omega)P$. Any two nonzero
	differentials have linearly equivalent divisors; their common class
	$K_{\mathcal X}$ is the \emph{canonical class} of $\mathcal X$, and
	$\deg(K_{\mathcal X})=2g-2$. The \emph{Riemann-Roch Theorem} states
	that for every $G\in\mathrm{Div}(\mathcal X)$,
	\begin{equation}
		\label{eq:RR}
		\ell(G)=\deg(G)+1-g+\ell(K_{\mathcal X}-G).
	\end{equation}
	The quantity $i(G):=\ell(K_{\mathcal X}-G)$ is the \emph{index of
		speciality} of $G$.
	
	\begin{definition}
		\label{def:nonspecial}
		A divisor $G\in\mathrm{Div}(\mathcal X)$ is \emph{non-special} if
		$i(G)=0$, equivalently $\ell(G)=\deg(G)+1-g$; otherwise $G$ is
		\emph{special}.
	\end{definition}
	
	By Riemann-Roch Theorem, every divisor of degree larger than $2g-2$ is
	automatically non-special and every non-special divisor has degree
	at least $g-1$. Non-special divisors of degree exactly $g-1$ or $g$
	are called \emph{non-special divisors of small degree}; these are exactly the divisors with $\ell(G)=0$ or
	$\ell(G)=1$, respectively, the smallest values compatible with
	non-speciality. These are precisely the divisors that play the
	central role in the LCP constructions of Section~\ref{sec:characterization}.
	
	We will also need the following divisors. For
	$A,B\in\mathrm{Div}(\mathcal X)$, define
	\[
	\gcd(A,B):=\sum_{P}\min\{v_P(A),v_P(B)\}P,
	\qquad
	\mathrm{lmd}(A,B):=\sum_P\max\{v_P(A),v_P(B)\}P.
	\]
	These two divisors satisfy the fundamental identities
	\begin{equation}
		\label{eq:gcdlmd}
		\mathcal{L}(A)\cap\mathcal{L}(B)=\mathcal{L}(\gcd(A,B))
		\quad \mbox{and} \quad
		\gcd(A,B)+\mathrm{lmd}(A,B)=A+B.
	\end{equation}
	
	\subsection{Algebraic Geometry Codes}
	\label{ssec:agcodes}
	
	Let $\mathcal X$ be as above and let $P_1,\dots,P_n\in\mathcal
	X(\mathbb F_q)$ be pairwise distinct $\mathbb F_q$-rational points.
	Set $D:=P_1+\cdots+P_n$, and let $G\in\mathrm{Div}(\mathcal X)$
	satisfy $\mathrm{Supp}(G)\cap\mathrm{Supp}(D)=\emptyset$.
	
	\begin{definition}
		\label{def:agcode}
		The
		\emph{algebraic geometry (AG) code}  $\CL(D,G)$ associated with the divisors $D$ and $G$ is defined as the image of the evaluation map 
		\begin{align*}
			\mathrm{ev}_D \colon  \mathcal L(G)& \longrightarrow \mathbb F_q^n \\
			z & \longmapsto  (z(P_1),\dots,z(P_n));
		\end{align*}
		that is,
		\[
		\CL(D,G)=\{(z(P_1),\dots,z(P_n))\mid z\in\mathcal L(G)\}\subseteq
		\mathbb F_q^n.
		\]
	\end{definition}
	
	\begin{proposition}[{\cite[Thm.~2.2.2]{Stichtenoth2009}}]
		\label{prop:agparams}
		The code $\CL(D,G)$ has parameters $[n,k,d]$, where
		\[
		k=\ell(G)-\ell(G-D) \quad \mbox{and} \quad d\ge n-\deg(G).
		\]
		In particular, if $2g-2<\deg(G)<n$, then
		$k=\ell(G)=\deg(G)+1-g$.
	\end{proposition}
	
	The dual of an AG code is again an AG code.
	
	\begin{proposition}[{\cite[Thm.~2.2.2, Prop.~8.1.2]{Stichtenoth2009}}]
		\label{prop:dualcode}
		Let $t\in\mathbb F_q(\mathcal X)$ be such that $v_{P_i}(t)=1$ for
		$1\le i\le n$, and set $\eta=dt/t$. Then $v_{P_i}(\eta)=-1$ with
		residue $1$ at each $P_i$ with $1\le i\le n$ and
		\[
		\CL(D,G)^{\perp}=\CL(D,H),
		\]
		where $H:=D-G+\mathrm{div}(\eta)$.
	\end{proposition}
	
	Finally, we recall the notion of code equivalence, which will be
	used to pass from an LCP to an LCD code by an explicit coordinate
	rescaling.
	
	\begin{definition}
		\label{def:equivcodes}
	Two codes $\C_1,\C_2\subseteq\mathbb F_q^n$ are  \emph{(linearly) equivalent} if
			there exists $\boldsymbol{a}=(a_1,\dots,a_n) \in (\mathbb F_q^*)^n$ such that
			$\C_2= \boldsymbol{a} \cdot \C_1:=\{(a_1c_1,\dots,a_nc_n)\mid(c_1,\dots,c_n)\in
			\C_1\}$. Equivalent codes have the same length, dimension and
			minimum distance. If $G\sim H$ as divisors with
			$\mathrm{Supp}(G)\cap\mathrm{Supp}(D)=\mathrm{Supp}(H)\cap
			\mathrm{Supp}(D)=\emptyset$, say $G=H+\mathrm{div}(a)$ with
			$v_{P_i}(a)=0$ for all $i$, then $\CL(D,H)=\boldsymbol{a}\cdot \CL(D,G)$ for
			$\boldsymbol{a}=(a(P_1),\dots,a(P_n))$.
	\end{definition}
	
	\subsection{Linear Complementary Pairs and LCD Codes}
	\label{ssec:lcp}
	
	\begin{definition}
		\label{def:lcp}
		Let $\C_1$ and $\C_2$ be linear codes of length $n$ over
		$\mathbb{F}_q$. The pair $(\C_1,\C_2)$ is called a
		\emph{linear complementary pair} (LCP) if
		\[
		\C_1+\C_2=\mathbb{F}_q^n
		\qquad\text{and}\qquad
		\C_1\cap \C_2={0}.
		\]
		Equivalently,
		\[
		\C_1\oplus \C_2=\mathbb{F}_q^n.
		\]
		A linear code $\C$ is called a \emph{linear complementary dual}
		(LCD) code if
		\[
		\C\cap \C^\perp={0},
		\]
		or equivalently,
		\[
		\C\oplus\C^\perp=\mathbb{F}_q^n.
		\]
		Hence LCD codes arise as a particular case of LCPs by taking
		$\C_2=\C_1^\perp$.
	\end{definition}
	
	LCD codes were introduced by Massey \cite{Massey1992} and
	have since attracted considerable attention due to their
	applications in coding theory and cryptography. More
	generally, LCPs of codes play a central role in Direct Sum
	Masking (DSM), a countermeasure against side-channel attacks
	(SCA) and fault-injection attacks (FIA)
	\cite{Bringer14,Carlet2018}. For an LCP $(\C_1,\C_2)$, the
	minimum distance $d(\C_1)$ measures the resistance against
	FIA, whereas the minimum distance $d(\C_2^\perp)$ measures
	the resistance against SCA. The quantity
	\[
	d(\C_1,\, \C_2) := \min\{d(\C_1),\, d(\C_2^\perp)\}
	\]
	is called the \emph{security parameter} of the pair
	\cite{Carlet2018}. In the particular case of an LCD code
	$\C$, the security parameter reduces to $d(\C)$.

	Let $P_1,\dots,P_n$ and $D$ be as in
	Section~\ref{ssec:agcodes}, and let $G,H\in\mathrm{Div}(\mathcal
	X)$ satisfy $\mathrm{Supp}(G)\cap\mathrm{Supp}(D)=
	\mathrm{Supp}(H)\cap\mathrm{Supp}(D)=\emptyset$. In \cite{Bhowmick2024}, Bhowmick, Dalai, and Mesnager provided sufficient conditions on the divisors $G$, $H$ and $D$ to ensure that a pair $(\CL(D,G),\CL(D,H))$ is an LCPs of AG codes, which relies on the explicit understanding of non-special divisors of degree $g-1$.
	\begin{lemma}[{\cite[Theorem.~3.5]{Bhowmick2024}}]
		\label{lem:lcpcriterion}
		Let $\mathcal X/\mathbb F_q$ have genus $g\ge 1$. Suppose that
		$2g-2 < \deg(G),\, \deg(H)<n$ and that the divisors $G$ and $H$ satisfy
		\begin{enumerate}[label=(\roman*)]
			\item  $\ell(G)+\ell(H)=n$,
			\item $\deg(\gcd(G,H))=g-1$, and
			\item both $\gcd(G,H)$ and $\mathrm{lmd}(G,H)-D$ are
			non-special.
		\end{enumerate}
		Then $(\CL(D,G), \, \CL(D,H))$ is an LCP of AG codes.
	\end{lemma}
	
	By Lemma~\ref{lem:lcpcriterion}, the construction of LCPs of AG
	codes hinges entirely on producing, for suitable $G,H$, an explicit
	non-special divisor $\gcd(G,H)$ of degree $g-1$ together with a
	non-special divisor $\mathrm{lmd}(G,H)-D$. The existence of
	non-special divisors of degree $g$ and $g-1$ is guaranteed in
	general by Ballet and Le Brigand \cite{Ballet2006} for $q\ge 3$ (resp.\
	$q\ge 4$), but an \emph{explicit} description of such divisors
	amounts, geometrically, to deciding whether a given class in
	$\mathrm{Pic}^{g-1}(\mathcal X)$ avoids the theta divisor of the
	Jacobian $J(\mathcal X)$, a problem with no algorithmic solution
	for a general curve. The remainder of the paper resolves this
	problem completely for hyperelliptic curves by translating it, via
	the Mumford representation recalled below, into a condition on
	polynomial degrees.
	
	\subsection{Hyperelliptic Curves and the Mumford Representation}
	\label{ssec:hyperelliptic}
	
	Let $g\ge 2$ and let $\mathcal X$ be a hyperelliptic curve of genus
	$g$ over $\mathbb F_q$. Regardless of $\mathrm{char}(\mathbb F_q)$,
	$\mathcal X$ admits an affine (imaginary quadratic) model
	\begin{equation}
		\label{eq:hyperelliptic}
		\mathcal X:\quad y^2+h(x)y=f(x),
	\end{equation}
	where $h,f\in\mathbb F_q[x]$, $f$ is monic of degree $2g+1$,
	$\deg(h)\le g$, and there is no point
	$(x,y)\in\overline{\mathbb F}_q\times\overline{\mathbb F}_q$
	satisfying simultaneously
	\[
	y^2+h(x)y-f(x)=0,\qquad h'(x)y - f'(x)=0,\qquad 2y+h(x)=0
	\]
	(non-singularity). The curve $\mathcal X$ has a single
	$\mathbb F_q$-rational point at infinity $P_\infty$, and
	$\mathbb F_q(\mathcal X)=\mathbb F_q(x,y)$.
    
	The \emph{hyperelliptic involution} is the $\mathbb F_q$-automorphism
	\begin{align*}
		\iota \colon \mathcal X & \longrightarrow \mathcal X \\
		(x,y) &\longmapsto (x,-y-h(x)) \\ 
		P_\infty &\longmapsto P_\infty.
	\end{align*}
	
	A point $P$ with $\iota(P)=P$ is called a \emph{ramification point}
	of the degree-two cover $x:\mathcal X\to\mathbb P^1$.
	
	\begin{definition}
		\label{def:semireduced}
		A divisor $D\in\mathrm{Div}(\mathcal X)$ is \emph{semi-reduced} if it can be written as
		\[
		D=\sum_{i=1}^r m_iP_i + n P_\infty,
		\]
		where $n \in \mathbb{Z}$, $m_i\ge0$, the affine points $P_i=(x_i,y_i)\ne P_\infty$ are
		pairwise distinct with $P_i\ne\iota(P_j)$ for $i\ne j$, and $m_i=1$
		whenever $P_i$ is a ramification point. A semi-reduced divisor is
		\emph{reduced} if, in addition, $\sum_{i=1}^r m_i\le g$.
		
		In particular, if a semi-reduced divisor $D$ has degree $0$, then $n = -\sum_{i=1}^r m_i$, meaning it takes the form
		\[
		D=\sum_{i=1}^r m_iP_i-\Big(\sum_{i=1}^r m_i\Big)P_\infty.
		\]
	\end{definition}
	
	While Definition \ref{def:semireduced} applies to divisors of arbitrary degrees, our primary interest lies in the Jacobian of the curve, which consists of divisor classes of degree zero. By the Riemann-Roch Theorem, every class in  $\mathrm{Pic}^0(\mathcal X)\cong J(\mathcal X)$ has a unique representative of the form $D_{\mathrm{aff}}-dP_\infty$ with $D_{\mathrm{aff}}$ being an effective reduced divisor of degree $d\le g$; equivalently, every divisor of degree zero on $\mathcal X$ is linearly equivalent to a reduced divisor of this form. To manipulate such classes effectively, we focus on these degree $0$ semi-reduced divisors and encode them as pairs of univariate polynomials, following Mumford \cite{Mum84}.
	
	\begin{lemma}[Mumford Representation, {\cite{Mum84}}]
		\label{lem:mumford}
		Every semi-reduced divisor
		$D=\sum_{i=1}^r m_iP_i-(\sum_{i=1}^r m_i)P_\infty$ on $\mathcal X$
		is uniquely represented by a pair of polynomials
		$(u,v)\in\mathbb F_q[x]\times\mathbb F_q[x]$ such that:
		\begin{enumerate}
			\item[1)] $u$ is monic, with $u(x)=\prod_{i=1}^r(x-x_i)^{m_i}$;
			\item[2)] $u(x)\mid v(x)^2+h(x)v(x)-f(x)$;
			\item[3)] $\deg(v)<\deg(u)$, where $v$ interpolates the
			$y$-coordinates of $\mathrm{Supp}(D)$ with the appropriate
			multiplicities, i.e.\ $v(x_i)=y_i$.
		\end{enumerate}
		Moreover, $D$ is reduced if and only if $\deg(u)\le g$.
	\end{lemma}
	
	We write $D\leftrightarrow(u,v)$, or $D=D_{(u,v)}$, when $(u,v)$ is the Mumford representation of the semi-reduced divisor $D$, and call such a pair a \emph{Mumford pair} on $\mathcal X$; it is \emph{reduced} when $\deg(u)\le g$. By Lemma~\ref{lem:mumford}, every class in $J(\mathcal X)$ corresponds to a unique reduced Mumford pair, and the group law of $J(\mathcal X)$ can be carried out directly on Mumford pairs via Cantor's algorithm \cite{Cantor1987}, which implements the reduction step by repeatedly replacing conjugate pairs $P+\iota(P)$ in a semi-reduced divisor by $2P_\infty$, until the affine part has degree at most $g$. This polynomial encoding is what allows us, in Section~\ref{sec:code_constructions}, to replace the (generally intractable) verification of non-speciality by an explicit computation of $\deg(u)$.

	\begin{remark}
		\label{rem:maximal}
		A function field $\mathcal X/\mathbb F_{q^2}$ of genus $g$ is
			\emph{maximal} if it attains the Hasse-Weil upper bound
			$\#\mathcal X(\mathbb F_{q^2})=q^2+1+2gq$. For a maximal curve, the
			Frobenius endomorphism of $J(\mathcal X)$ acts as multiplication by
			$-q$, which forces
			$J(\mathcal X)(\mathbb F_{q^2})\cong(\mathbb Z/(q+1)\mathbb Z)^{2g}$
			as a group; in particular $\#J(\mathcal X)(\mathbb F_{q^2})=(q+1)^{2g}$.
			This fact is used in Section~\ref{sec:code_constructions} to obtain
			explicit families of MDS LCD codes on maximal hyperelliptic curves.
	\end{remark}

	\section{Characterization of Non-Special Divisors via Mumford Representations}
	\label{sec:characterization}
	
	In this section, we present our main results concerning the algebraic characterization of non-special divisors of small degree on hyperelliptic curves. Geometrically, identifying whether a divisor $A$ of degree $g-1$ or $g$ is non-special (i.e., satisfies $\ell(A)=0$ or $\ell(A)=1$, respectively) is a deeply classical problem in algebraic geometry, as it corresponds to determining whether its linear equivalence class avoids the subvarieties of the Jacobian, such as the Theta divisor $\Theta$. 
	
	Using the bijection between divisor classes and reduced polynomial pairs established in Section~\ref{sec:preliminaries}, we demonstrate that this geometric property translates into a direct constraint on the degree of the Mumford representative. This approach bypasses direct dimension computations and serves as the basis for our subsequent code constructions.
	
	We begin by giving a necessary and sufficient criterion for a hyperelliptic divisor of degree $g-1$ to be non-special.
	
	\begin{proposition}
		\label{prop:l(A)=0_iff_degu=g}
		Let $A$ be a rational divisor on $\mathcal X$ of degree $g-1$. Let $(u_A,v_A)$
		be the reduced Mumford representative of the divisor class
		$[A-(g-1)P_\infty]\in J(\mathcal X)$. Then $\ell(A)=0$ if and only if
		$\deg u_A=g$.
	\end{proposition}
	
	\begin{proof}
		Let $[D-dP_\infty]$ be the unique reduced representative of the class
		$[A-(g-1)P_\infty]$, where $D$ is an effective affine reduced divisor and
		$d=\deg D=\deg u_A\leq g$. Since
		$A-(g-1)P_\infty\sim D-dP_\infty$, we have
		$A\sim D+(g-1-d)P_\infty$.
		If $d\leq g-1$, then $D+(g-1-d)P_\infty$ is effective, so $A$ is linearly
		equivalent to an effective divisor and $\ell(A)\geq 1$. 
		
		Conversely, suppose $\ell(A)\geq 1$. Then there exists an effective divisor $E$ of degree $g-1$ such that $A \sim E$, which implies:
		\[
		A - (g-1)P_\infty \sim E - (g-1)P_\infty.
		\]
		We now consider the effective divisor $E$ to find its reduced representative:
		\begin{itemize}
			\item If $E$ is already reduced and does not contain $P_\infty$ in its support, then $E - (g-1)P_\infty$ is already the unique reduced Mumford representative. Thus, $D = E$ and $\deg u_A = \deg E = g-1$.
			\item If $E$ is not reduced or contains $P_\infty$, it must contain either $P_\infty$ or a pair of conjugate points $P + \iota(P)$. Using the hyperelliptic relation $P + \iota(P) \sim 2P_\infty$, we can repeatedly replace any conjugate pair with $2P_\infty$ and subtract any copies of $P_\infty$ from the affine part. Each such reduction step strictly decreases the degree of the affine part by at least $1$.
		\end{itemize}
		Since the reduction process can only decrease or maintain the degree of the effective affine part, and we began with $\deg E = g-1$, the final unique reduced divisor $D$ must satisfy:
		\[
		\deg u_A = \deg D \leq g-1.
		\]
		Therefore, $\ell(A)\geq 1$ if and only if $\deg u_A\leq g-1$. Since every reduced Mumford representative satisfies $\deg u_A\leq g$ by definition, it follows by contraposition that $\ell(A)=0$ if and only if $\deg u_A=g$.
		
	\end{proof}
	
	\begin{corollary}
		\label{cor:non-special-degree-g-minus-one-mumford}
		Let $A$ be a rational divisor on $\mathcal X$ of degree $g-1$, and let $(u_A,v_A)$ be the reduced Mumford representative of the class $[A-(g-1)P_\infty]\in J(\mathcal X)$. Then $A$ is non-special if and only if $\deg u_A=g$.
	\end{corollary}
	
	\begin{proof}
		Since $\deg A=g-1$, Riemann-Roch gives $\ell(A)-\ell(K-A)=0$. Hence $A$ is non-special, that is, $\ell(K-A)=0$, if and only if $\ell(A)=0$. The result follows from Proposition~\ref{prop:l(A)=0_iff_degu=g}.
	\end{proof}
	
	We shall also use the analogous criterion for divisors of degree $g$.
	
	\begin{theorem}
		\label{thm:non-special-degree-g-mumford}
		Let $A$ be a rational divisor on $\mathcal X$ of degree $g$. Let $(u_A,v_A)$ be
		the reduced Mumford representative of the divisor class $[A-gP_\infty]$.
		Then $A$ is non-special if and only if $\deg u_A\in\{g-1,g\}$.
		Equivalently, $A$ is non-special if and only if $\deg u_A\geq g-1$.
	\end{theorem}
	
	\begin{proof}
		Let $[D-dP_\infty]$ be the unique reduced representative of the class
		$[A-gP_\infty]$, where $D$ is an effective affine reduced divisor and
		$d=\deg D=\deg u_A\leq g$. Since $A-gP_\infty\sim D-dP_\infty$, we have
		$A\sim D+(g-d)P_\infty$. Let $s=g-d$. We use the following formula for dimension of Riemann-Roch spaces associated to semi-reduced divisors on hyperelliptic curves (see \cite[Lemma 2.2.5]{MBoerThesis}):
		if $D$ is effective affine semi-reduced of degree $d$, then
		\[
		\ell(D+sP_\infty)=
		\begin{cases}
			\left\lfloor s/2\right\rfloor+1, & \text{if } 2d+s\leq 2g-2,\\
			d+s+1-g, & \text{if } 2d+s>2g-2.
		\end{cases}
		\]
		In the present situation $s=g-d$, so $2d+s=g+d$. Hence the first case occurs
		exactly when $d\leq g-2$.
		
		If $d\leq g-2$, then $s\geq 2$, and the formula gives
		$\ell(A)=\ell(D+sP_\infty)\geq 2$. Since $\deg A=g$, Riemann-Roch gives
		$\ell(A)=1+\ell(K-A)$, so $\ell(K-A)>0$ and $A$ is special. Conversely, if
		$d\geq g-1$, then $d\in\{g-1,g\}$ and the second case applies. Since $d+s=g$,
		we get $\ell(A)=d+s+1-g=1$. Again using $\ell(A)=1+\ell(K-A)$, we obtain
		$\ell(K-A)=0$, and therefore $A$ is non-special.
	\end{proof}
	\begin{example}
		We now apply the algebraic criteria to explicitly construct non-special divisors of degree $g$ and $g-1$ over imaginary hyperelliptic curves. Throughout this example, let $P_\infty$ denote the unique ramified point at infinity. Recall that an effective affine divisor $D$ is reduced if its support avoids $P_\infty$ and contains no conjugate pairs $P + \iota(P)$.
			
			\textbf{Non-special divisors of degree $g$:}
			By Theorem \ref{thm:non-special-degree-g-mumford}, a divisor $A$ of degree $g$ is non-special if and only if its reduced affine part has degree $g$ or $g-1$. This provides a straightforward construction method, as follows.
			\begin{itemize}
				\item Let $P_1, \dots, P_g$ be pairwise distinct affine rational places such that $P_i \neq \iota(P_j)$ for all $i,j$. Then $A = \sum_{i=1}^g P_i$ is a non-special divisor of degree $g$.
				\item If we instead consider $g-1$ rational places satisfying the same condition, the divisor $A = P_\infty + \sum_{i=1}^{g-1} P_i$ is also non-special, as its reduced affine part has degree $g-1$.
				\item This naturally extends to places of higher degree. For instance, if $Q$ is an affine  place of degree $g$ containing no conjugate pairs in its support, then $A = Q$ is non-special. If $\deg Q = g-1$, we may take $A = Q + P_\infty$.
				\item More generally, any combination of places summing to degree $g$ or $g-1$ yields a non-special divisor (adjusting with $P_\infty$ in the latter case), provided the overall support remains strictly reduced.
			\end{itemize}
			
			\textbf{Non-special divisors of degree $g-1$:}
			For degree $g-1$, Corollary \ref{cor:non-special-degree-g-minus-one-mumford} requires the reduced affine part to have degree exactly $g$. A direct approach is to take a reduced effective affine divisor of degree $g$ and subtract $P_\infty$:
			\begin{itemize}
				\item Considering the previous rational places $P_1, \dots, P_g$, the divisor $A = \sum_{i=1}^g P_i - P_\infty$ is non-special of degree $g-1$.
				\item Similarly, taking the closed place $Q$ of degree $g$ defined above, we obtain the non-special divisor $A = Q - P_\infty$.
				\item Any combination of closed and rational places forming a reduced effective affine divisor $D$ of degree $g$ immediately provides a non-special divisor $A = D - P_\infty$.
		\end{itemize}
	\end{example}

	\section{Construction of LCP of MDS Codes and LCD Codes}
	\label{sec:code_constructions}
	
	In this section, we use the algebraic characterizations of non-special divisors from Section~\ref{sec:characterization} to provide a systematic framework for constructing geometric codes. By replacing computationally expensive Riemann-Roch dimension computations with explicit polynomial degree conditions, we obtain efficient constructions with cryptographic applications. First, we detail a polynomial-based construction of Maximum Distance Separable (MDS) Linear Complementary Pairs (LCP) of codes. Second, we demonstrate that a 2-torsion condition on the Jacobian variety allows for the explicit computation of multiplier vectors to construct Linear Complementary Dual (LCD) codes.
	
	\subsection{LCP of MDS Codes via Mumford Representation}
	\label{subsec:lcp-mumford}
	
	We now describe the construction of our codes. The input data consists of two Mumford pairs: a reduced pair $(u_1,v_1)$ of degree $g$, and a semi-reduced pair $(u_2,v_2)$ of degree $t$. The corresponding divisors are subsequently defined in terms of these fixed polynomials."
	
	Let $t\geq g$. Choose a reduced Mumford pair $(u_1,v_1)$ with $u_1$ monic and $\deg u_1=g$, and choose a semi-reduced Mumford pair $(u_2,v_2)$ with $u_2$ monic and $\deg u_2=t$. Assume that $\gcd(u_1,u_2)=1$ and $\gcd(u_2,2v_2+h)=1$. Define $v_G$ and $v_H$ modulo $u_1 u_2$ by the Chinese Remainder conditions
	\[
	v_G\equiv v_1 \pmod{u_1}, \qquad v_G\equiv v_2 \pmod{u_2},
	\]
	and
	\[
	v_H\equiv v_1 \pmod{u_1}, \qquad v_H\equiv -v_2-h \pmod{u_2}.
	\]
	Based on these Mumford pairs, we define the following divisors $A:=D_{(u_1,v_1)} - P_\infty$, $B:=D_{(u_2,v_2)}$, $G:=D_{(u_1 u_2,v_G)} - P_\infty$ and $H:=D_{(u_1 u_2,v_H)}-P_\infty$. 
	
	From a divisor-theoretic perspective, the following proposition illustrates the consequences of the polynomial conditions on the associated Mumford pairs.
	
	\begin{proposition}
		\label{prop:divisorial-consequences-mumford-data}
		With the notation above, the divisor $A=D_{(u_1,v_1)}-P_\infty$ is non-special of degree $g-1$ and $\gcd(G,H)=A$. Moreover, if $\mathcal D=\sum_{i=1}^t(P_i+\iota(P_i))$ is a sum of $t$ conjugate pairs, then $\operatorname{lmd}(G,H)-\mathcal D\sim A$.
	\end{proposition}
	
	\begin{proof}
		Since $(u_1,v_1)$ is a reduced Mumford pair with $\deg u_1=g$, the class $[D_{(u_1,v_1)}-gP_\infty]$ is represented by $(u_1, v_1)$. Note that $A=D_{(u_1,v_1)}-P_\infty$ is a divisor of degree $g-1$ and $D_{(u_1,v_1)}$ is a representative of $[A-(g-1)P_\infty]$ in $J(\X)$. By Proposition~\ref{prop:l(A)=0_iff_degu=g}, $\ell(D_{(u_1,v_1)}-P_\infty)=0$, which means that $A=D_{(u_1,v_1)}-P_\infty$ is a non-special divisor of degree $g-1$.
		
		By the congruences defining $v_G$ and $v_H$, the affine divisors corresponding to $G$ and $H$ decompose as $D_G=D_{(u_1,v_1)}+B$ and $D_H=D_{(u_1,v_1)}+\iota(B)$, respectively. The condition $\gcd(u_2,2v_2+h)=1$ guarantees that the conjugate divisors $B$ and $\iota(B)$ are disjoint. Thus, their common affine part is exactly $D_{(u_1,v_1)}$. Since both $G$ and $H$ have coefficient $-1$ at $P_\infty$, it follows that $\gcd(G,H)=D_{(u_1,v_1)}-P_\infty=A$.
		
		Finally, we have $\operatorname{lmd}(G,H)=D_A+B+\iota(B)-P_\infty$. Because any sum of $t$ conjugate pairs on the curve is linearly equivalent to $2tP_\infty$, we obtain 
		\[
		\cD = \sum_{i=1}^t (P_i + \iota(P_i)) \sim \sum_{i=1}^t 2P_\infty = 2tP_\infty \sim B + \iota(B),
		\]
		Substituting this equivalence immediately yields $\operatorname{lmd}(G,H)-\mathcal D \sim D_A-P_\infty=A$.
	\end{proof}
	
	We are now ready to describe families of LCP of codes obtained from hyperelliptic curves and the Mumford representation.
	
	\begin{theorem}
		\label{thm:LCP-hyperelliptic-Goppa}
		Keep the polynomial Mumford data above. Let $\mathcal D=\sum_{i=1}^t(P_i+\iota(P_i))$ be an evaluation divisor supported on $2t$ distinct rational points, disjoint from $\operatorname{Supp}(G)\cup\operatorname{Supp}(H)$. Assume also that $u_A u_B$ is coprime to the polynomial vanishing on the $x$-coordinates of $\operatorname{Supp}(\mathcal D)$. Then, the AG codes $\CL(\mathcal D,G)$ and $\CL(\mathcal D,H)$ form a linear complementary pair of length $2t$ and dimension $t$. Moreover, the minimum distance of each code satisfies $d\geq t-g+1$.
	\end{theorem}
	
	\begin{proof}
		The coprimality condition guarantees that the supports of $G$ and $H$ are disjoint from $\operatorname{Supp}(\mathcal D)$, making the evaluation maps well-defined. 
		
		By Proposition~\ref{prop:divisorial-consequences-mumford-data}, we have $\gcd(G,H)=A$ and $\operatorname{lmd}(G,H)-\mathcal D\sim A$. Since $\ell(A)=0$, it immediately follows that $\ell(\gcd(G,H))=0$ and $\ell(\operatorname{lmd}(G,H)-\mathcal D)=0$. Because $G, H \leq \operatorname{lmd}(G,H)$, this latter equality also forces $\ell(G-\mathcal D)=\ell(H-\mathcal D)=0$, ensuring that both evaluation maps are injective.
		
		To see that $\CL(\mathcal D,G) \cap \CL(\mathcal D,H) = \{0\}$, suppose $c=\operatorname{ev}_{\mathcal D}(f)=\operatorname{ev}_{\mathcal D}(h)$ is a codeword in the intersection, with $f\in\CL(G)$ and $h\in\CL(H)$. Then $f-h$ vanishes on $\mathcal D$, meaning $f-h \in \mathcal{L}(\operatorname{lmd}(G,H)-\mathcal D)=\{0\}$. Thus $f=h$, which implies $f \in \mathcal{L}(\gcd(G,H))=\{0\}$. Hence $c=0$.
		
		Finally, $\deg G = \deg H = \deg(u_1 u_2)-1 = g+t-1$. Since $t\geq g$, we have $\deg G > 2g-2$, and the Riemann-Roch theorem yields $\dim\CL(\mathcal D,G) = \dim\CL(\mathcal D,H) = t$. Because their dimensions sum to the code length $2t$ and their intersection is trivial, $\CL(\mathcal D,G) \oplus \CL(\mathcal D,H) = \mathbb{F}_q^{2t}$, forming a linear complementary pair. The Goppa bound guarantees the minimum distance is at least $2t-\deg G = t-g+1$ for both codes.
	\end{proof}

    To ease the notation and emphasize their connection with the divisors $A$ and $B$, from now on we will denote the Mumford pairs $(u_1, v_1)$ and $(u_2, v_2)$ by $(u_A, v_A)$ and $(u_B, v_B)$, respectively.
    
	\begin{remark}
The polynomials $u_A, v_A, u_B$, and $v_B$ can be computed explicitly using standard arithmetic over $\mathbb F_q[x]$ and its extensions.
			
\begin{enumerate}
	\item \textbf{The $u$-polynomials:} To maximize the number of rational points available for the evaluation divisor $\mathcal D$, one may choose $u_A$ and $u_B$ to be irreducible polynomials in $\mathbb F_q[x]$ of degrees $g$ and $t$, respectively. This ensures that $\gcd(u_A, u_B) = 1$ and that their roots, which lie in the extensions $\mathbb F_{q^g}$ and $\mathbb F_{q^t}$, are disjoint from the $\mathbb F_q$-rational support of $\mathcal D$. Alternatively, if $\mathcal X(\mathbb F_q)$ is sufficiently large, one may define $u_A$ and $u_B$ simply by selecting $g+t$ distinct rational points outside of $\mathcal D$ and taking the product of their corresponding linear factors $x - x_i$.
				
	\item \textbf{The $v$-polynomials:} Finding $v_A$ amounts to solving the congruence $v_A^2 + hv_A - f \equiv 0 \pmod{u_A}$ with $\deg v_A < \deg u_A$. If $u_A$ is irreducible, the quotient $\mathbb F_q[x]/\langle u_A \rangle$ is isomorphic to the finite field $\mathbb F_{q^g}$. Let $\delta \in \mathbb F_{q^g}$ be a root of $u_A$. Solving the congruence is equivalent to finding a root $\xi \in \mathbb F_{q^g}$ of the quadratic equation $Y^2 + h(\delta)Y - f(\delta) = 0$. The polynomial $v_A \in \mathbb F_q[x]$ is then the unique preimage of $\xi$ under the evaluation isomorphism $\mathbb F_q[x]/\langle u_A \rangle \xrightarrow{\sim} \mathbb F_{q^g}$ given by $x \mapsto \delta$. If $u_A$ splits into distinct linear factors over $\mathbb F_q$, then $v_A$ is simply the unique Lagrange interpolating polynomial satisfying $v_A(x_i) = y_i$, where $(x_i, y_i) \in \mathcal X(\mathbb F_q)$. The polynomial $v_B$ is constructed analogously modulo $u_B$, ensuring that $\gcd(u_B, 2v_B+h)=1$ by avoiding ramification points.
				
	\item \textbf{The Chinese Remainder Theorem (CRT):} Since $\gcd(u_A, u_B) = 1$, the Extended Euclidean Algorithm yields polynomials $a, b \in \mathbb F_q[x]$ such that $a u_A + b u_B = 1$. The polynomials $v_G$ and $v_H$ are given by the standard ring isomorphism 
	\[
	\mathbb F_q[x]/\langle u_A u_B \rangle \xrightarrow{ \ \sim \ } \mathbb F_q[x]/\langle u_A \rangle \times \mathbb F_q[x]/\langle u_B \rangle,
	\]
	yielding the explicit formulas:
	\[
	v_G \equiv v_B a u_A + v_A b u_B \pmod{u_A u_B}
	\]
	and
	\[
	v_H \equiv (-v_B - h) a u_A + v_A b u_B \pmod{u_A u_B}.
	\]
	\end{enumerate}
	\end{remark}
	
	\begin{proposition}
		\label{prop:lcp-y2+y=x^q+1}
		Let $q>2$ be an even prime power and let $\mathcal{X}$ be the hyperelliptic curve defined by $y^2 + y = x^{q+1}$ over $\mathbb{F}_{q^2}$, with a unique point at infinity $P_\infty$. Let $k$ be a divisor of $q-1$, set $t = kq$, and define the evaluation polynomial $p_{\mathcal{D}}(x) = (x^q+x)^k - 1$. Let $(u_A,v_A)$ and $(u_B,v_B)$ be Mumford pairs determining effective affine divisors $D_A$ and $B$, respectively. Assume that:
		\begin{enumerate}[label=(\roman*)]
			\item $(u_A, v_A)$ is reduced, with $u_A$ monic of degree $g$;
			\item $(u_B, v_B)$ is semi-reduced, with $u_B$ monic of degree $kq$;
			\item $\gcd(u_A, u_B) = 1$;
			\item $\gcd(u_Au_B, x \cdot p_{\mathcal{D}}(x)) = 1$.
		\end{enumerate}
		
		For the evaluation divisor $\mathcal{D} = \operatorname{div}_0(p_{\mathcal{D}}(x))$ and the code divisors $G = D_A + B - P_\infty$ and $H = D_A + \iota(B) - P_\infty$, the algebraic geometry codes $\mathcal{C}_\mathcal{L}(\mathcal{D}, G)$ and $\mathcal{C}_\mathcal{L}(\mathcal{D}, H)$ form a linear complementary pair of length $2t$ and dimension $t$, with minimum distances $d \ge t - g + 1$.
	\end{proposition}
	
	\begin{proof}
		We apply Theorem \ref{thm:LCP-hyperelliptic-Goppa}. First, observe that the roots of $p_{\mathcal{D}}(x) = (x^q+x)^k - 1$ correspond to elements $\alpha \in \mathbb{F}_{q^2}$ such that the trace $\operatorname{Tr}_{\mathbb{F}_{q^2}/\mathbb{F}_q}(\alpha) = \alpha^q + \alpha$ is a $k$-th root of unity in $\mathbb{F}_q^*$. Since $k \mid q-1$, there are exactly $k$ such distinct roots of unity in $\mathbb{F}_q^*$. So, it yields exactly $t = kq$ distinct $x$-coordinates in $\mathbb{F}_{q^2}$.
		Because $q$ is even, the hyperelliptic involution $\iota(x, y) = (x, y+1)$ acts without fixed points in the affine plane. Consequently, $\mathcal{D}$ is an evaluation divisor supported on $2t$ distinct rational points and can be written as $\sum_{i=1}^t (P_i + \iota(P_i))$.
		
		Next, the polynomial vanishing on the $x$-coordinates of $\operatorname{Supp}(\mathcal{D})$ is precisely $p_{\mathcal{D}}(x)$. Condition (iv) guarantees that $\gcd(u_Au_B, p_{\mathcal{D}}(x)) = 1$, which ensures that the supports of $G$ and $H$ are disjoint from $\operatorname{Supp}(\mathcal{D})$. 
		
		The remaining hypotheses of Theorem \ref{thm:LCP-hyperelliptic-Goppa} are immediately satisfied by conditions (i), (ii), and (iii) along with the definitions of $G$ and $H$. Thus, it follows that $\mathcal{C}_\mathcal{L}(\mathcal{D}, G)$ and $\mathcal{C}_\mathcal{L}(\mathcal{D}, H)$ form an LCP of length $2t$ and dimension $t$ and the minimum distance of each code satisfies $d \ge kq - g + 1$, by Proposition \ref{prop:agparams}.
	\end{proof}

	We now provide a practical criterion to decide when this LCP consists of MDS codes. Write $\operatorname{Supp}(\mathcal D)=\{R_1,\ldots,R_{2t}\}$. For each subset $I\subseteq\{1,\ldots,2t\}$ of size $t$, define the Jacobian class $\gamma_I=[\sum_{i\in I}R_i-tP_\infty]\in J(\mathcal X)$. Additionally, let $\alpha=[D_A-gP_\infty]$ and $\beta=[B-tP_\infty]$.
	
	\begin{proposition}
		\label{prop:practical-theta-avoidance}
		With the notation above, the codes $\CL(\mathcal D,G)$ and $\CL(\mathcal D,H)$ are MDS if and only if, for every subset $I\subseteq\{1,\ldots,2t\}$ of size $t$, the reduced Mumford representatives of the classes $\alpha+\beta-\gamma_I$ and $\alpha-\beta-\gamma_I$, respectively, have $u$-polynomials of degree $g$.
	\end{proposition}
	
	\begin{proof}
		By the MDS criterion for AG codes of length $2t$ and dimension $t$, $\CL(\mathcal D,G)$ is MDS if and only if $\ell(G-\sum_{i\in I}R_i)=0$ for every subset $I$ of size $t$. Since $\deg(G-\sum_{i\in I}R_i)=g-1$, Proposition~\ref{prop:l(A)=0_iff_degu=g} states that this dimension is zero if and only if the degree zero class $[G-\sum_{i\in I}R_i-(g-1)P_\infty]$ has a Mumford $u$-polynomial of degree $g$.
		
		Recalling that $G = D_A + B - P_\infty$, we can expand and group the terms of this class as
		\[
		\left[D_A+B-\sum_{i\in I}R_i-gP_\infty\right] = [D_A-gP_\infty] + [B-tP_\infty] - \left[\sum_{i\in I}R_i-tP_\infty\right] = \alpha+\beta-\gamma_I.
		\]
		
		The argument for $\CL(\mathcal D,H)$ is identical, since $H = D_A + \iota(B) - P_\infty$ and $B+\iota(B)\sim 2tP_\infty$, we have $[\iota(B)-tP_\infty] = -[B-tP_\infty] = -\beta$. Substituting this into the same expansion yields the class $\alpha-\beta-\gamma_I$.
	\end{proof}
	
	\begin{corollary}
		\label{cor:theta-avoidance-two-torsion}
		Assume, in addition, that $2\alpha=0$ in $J(\mathcal X)$. Equivalently, for the Mumford pair $(u_A,v_A)$, assume that $u_A\mid 2v_A+h$. Then, if $\CL(\mathcal D,G)$ is MDS, the code $\CL(\mathcal D,H)$ is also MDS.
	\end{corollary}
	
	\begin{proof}
		Assume $\CL(\mathcal D,G)$ is MDS. Let $I \subseteq \{1,\ldots,2t\}$ be an arbitrary subset of size $t$, and let $I^c$ be its complement. Since $\mathcal D = \sum_{i=1}^{2t} R_i \sim 2tP_\infty$, we have $\gamma_I + \gamma_{I^c} = 0$, which means $\gamma_I = -\gamma_{I^c}$.
		
		Using the hypothesis $2\alpha=0$ (which implies $-\alpha=\alpha$), we relate the MDS condition for $H$ on the subset $I$ to the condition for $G$ on $I^c$ by taking the inverse in the Jacobian:
		\[
		-(\alpha-\beta-\gamma_I) = -\alpha+\beta+\gamma_I = \alpha+\beta-\gamma_{I^c}.
		\]
		Because $I^c$ also has size $t$, Proposition~\ref{prop:practical-theta-avoidance} guarantees that the reduced representative of $\alpha+\beta-\gamma_{I^c}$ has a $u$-polynomial of degree $g$. Taking the inverse of a Jacobian class preserves its $u$-polynomial, since the inverse of $(u,v)$ is $(u,-v-h\pmod u)$. Thus, the class $\alpha-\beta-\gamma_I$ must also have $u$-degree $g$. By Proposition~\ref{prop:practical-theta-avoidance}, $\CL(\mathcal D,H)$ is MDS.
	\end{proof}
	
	Combining the LCP construction with the theta-avoidance test gives the following form of the main construction.
	
	\begin{theorem}[MDS LCPs from Mumford data]
		\label{thm:MDS-LCP-Mumford-data}
		Keep the polynomial Mumford data $(u_A,v_A)$ and $(u_B,v_B)$, and let $G,H$ be the divisors induced by $(u_A u_B,v_G)$ and $(u_A u_B,v_H)$, respectively. Let $\mathcal D$ be an evaluation divisor supported on $2t$ distinct rational points forming $t$ conjugate pairs, disjoint from $\operatorname{Supp}(G)\cup\operatorname{Supp}(H)\cup\{P_\infty\}$.
		
		Write $\operatorname{Supp}(\mathcal D)=\{R_1,\ldots,R_{2t}\}$. Assume that, for every subset $I\subseteq\{1,\ldots,2t\}$ of size $t$, the reduced Mumford representatives of the classes
		\[
		\alpha+\beta-\gamma_I \quad\text{and}\quad \alpha-\beta-\gamma_I
		\]
		have $u$-polynomials of degree $g$. Then $\CL(\mathcal D,G)$ and $\CL(\mathcal D,H)$ form an MDS linear complementary pair. In particular, both codes have parameters
		\[
		[2t,t,t+1]_q.
		\]
	\end{theorem}
	
	\begin{proof}
		By Theorem~\ref{thm:LCP-hyperelliptic-Goppa}, the codes $\CL(\mathcal D,G)$ and $\CL(\mathcal D,H)$ form a linear complementary pair of length $2t$ and dimension $t$. Moreover, Proposition~\ref{prop:practical-theta-avoidance} ensures that the condition on the $u$-polynomial degrees is exactly the necessary and sufficient condition for both of these codes to be MDS. Finally, since an MDS code of length $n=2t$ and dimension $k=t$ meets the Singleton bound $d=n-k+1$, both codes achieve a minimum distance of $t+1$. Thus, they share the parameters $[2t,t,t+1]_q$.
	\end{proof}
	
	\begin{remark}
			The classical MDS conjecture imposes an immediate restriction on the parameters of LCPs of MDS codes. If $(\C_1, \C_2)$ is an LCP of non-trivial MDS codes over $\mathbb F_q$ with $\dim \C_i=k_i$, then $k_1+k_2=n$. Under the MDS conjecture, one must have $n\leq q+1$, except possibly when $q$ is even and $\{k_1,k_2\} = \{3,q-1\}$, in which case $n=q+2$ may occur. Consequently, LCP of MDS codes cannot exist for arbitrary lengths and complementary dimensions; their parameters are inherently constrained by the classical MDS conjecture applied to both components.
	\end{remark}
	\subsection{Explicit LCD Codes}
	\label{subsec:lcd_constructions}
	
	In this subsection, we focus on the explicit construction of Linear Complementary Dual(LCD) algebraic geometry codes. While the general theory of LCP codes considers the intersection of two distinct code spaces $\CL(\mathcal{D}, G)$ and $\CL(\mathcal{D}, H)$, a code $\mathcal{C}$ is LCD if and only if it forms an LCP with its own Euclidean dual, meaning $\mathcal{C} \cap \mathcal{C}^{\perp} = \{0\}$.
	
	To systematically achieve this property on hyperelliptic curves, we characterize the duals of our constructed AG codes by computing explicit residues of certain differential forms. By imposing a $2$-torsion condition in the Jacobian variety $J(\mathcal{X})$, which translates directly into a divisibility condition on the Mumford polynomials, we derive explicit multiplier vectors $\boldsymbol{\mu}$ that transform our evaluation codes into LCD codes. Finally, we apply this framework to maximal hyperelliptic curves over $\mathbb{F}_{q^2}$ to yield explicit families of MDS LCD codes with optimal parameters.
	
	To explicitly characterize the Euclidean dual of our codes, we rely on the classical regular differential 
	\[
	\omega = \frac{dx}{2y+h(x)},
	\]
	whose divisor is exactly
	\begin{equation}
		\label{eq:CanonicalDivHyp}
		K_{\mathcal{X}} \;=\; (2g-2)P_\infty
	\end{equation} 
	(see, e.g., \cite[Section 6.2]{Stichtenoth2009}). By combining residue computations of $\omega$ with a $2$-torsion condition on the Jacobian, we obtain the following explicit criterion for LCD codes.

\begin{theorem}
\label{thm:duality-two-torsion-lcd}
Keep the notation of Theorem~\ref{thm:LCP-hyperelliptic-Goppa}. Assume furthermore that
\[
2[D_A-gP_\infty]=0 \qquad\text{in }J(\mathcal X),
\]
which is equivalent to $u_A \mid (2v_A+h)$. Let $p_\mathcal{D}(x) = \prod_{i=1}^t (x - x_i)$ be the polynomial whose roots are the $x$-coordinates of the evaluation divisor $\mathcal D$. Define the vector $\mathbf{m} = (m_P)_{P \in \Supp(\mathcal D)}$ evaluated at each point $P = (x_P, y_P) \in \operatorname{Supp}(\mathcal D)$ by
\[
m_P = \frac{u_A(x_P) u_B(x_P)}{p_\mathcal{D}'(x_P) (2y_P + h(x_P))}.
\]
Then, the dual of $\CL(\mathcal D,G)$ is given by 
\[
\CL(\mathcal D,G)^\perp = \mathbf{m} \cdot \CL(\mathcal D,H) = \{ (m_P \cdot c_P)_{P \in \Supp(\mathcal D)} \mid \mathbf{c} \in \CL(\mathcal D,H) \}.
\]
Furthermore, if there exist elements $\mu_P \in \mathbb{F}_q^*$ such that $\mu_P^2 = m_P$ (which is always possible over a suitable extension or by careful choice of $x_i$), then the AG code 
\[
\mathcal{C} = \boldsymbol{\mu} \cdot \CL(\mathcal D, G) = \{ (\mu_P z(P))_{P \in \Supp(\mathcal D)} \mid z \in \mathcal{L}(G) \}
\]
is an LCD code.
\end{theorem}

\begin{proof}
By Proposition \ref{prop:dualcode}, the dual of $\CL(\mathcal D, G)$ is given by
\[
\CL(\mathcal D, G)^\perp=\mathcal{C}_\Omega(\mathcal{D}, G) = \CL(\mathcal{D}, \mathcal{D}+\operatorname{div}(\eta)-G).
\]
To compute this explicitly, we choose a differential $\eta$ with simple poles at $\mathcal{D}$. 

Consider the differential $\eta = \frac{dx}{p_\mathcal{D}(x)(2y+h)}$. 
Since the roots of $p_\mathcal{D}(x)$ correspond to the $x$-coordinates of $\mathcal D$ and $\mathcal D$ avoids ramification points, the principal divisor of $p_\mathcal{D}(x)$ is $\operatorname{div}(p_\mathcal{D}(x)) = \mathcal{D} - 2tP_\infty$. 
Thus, by \eqref{eq:CanonicalDivHyp}, the divisor of $\eta$ is given by
\[
\operatorname{div}(\eta) = \operatorname{div}\left(\frac{dx}{2y+h}\right) - \operatorname{div}(p_\mathcal{D}(x)) = (2g-2)P_\infty - (\mathcal{D} - 2tP_\infty) = (2g+2t-2)P_\infty - \mathcal D. 
\]
Since $x - x_P$ is a local parameter at $P$, $\eta$ has simple poles at all points $P = (x_P, y_P)$ of $\mathcal D$ with residues
\[
\operatorname{res}_P(\eta) = \frac{1}{p_\mathcal{D}'(x_P) (2y_P + h(x_P))}.
\]
Therefore, we have
\begin{align*}
\CL(\mathcal D, G)^\perp & = \CL(\mathcal D, \mathcal{D} + \operatorname{div}(\eta)  - G) \\
& = \operatorname{res}_{\mathcal{D}}(\eta) \cdot \CL(\mathcal D, (2g+2t-2)P_\infty - G)
\\ 
& =  \left\{ \ev_\mathcal{D}\left(  \frac{1}{p_\mathcal{D}'(x)(2y+h)}  \cdot \varphi \right) \mathrel{\Bigg|} \varphi \in \mathcal{L}((2g+2t-2)P_\infty - G) \right\},
\end{align*}
where $\operatorname{res}_{\mathcal{D}}(\eta) = (\operatorname{res}_P(\eta))_{P \in \Supp(\mathcal{D})}$.

We now determine the relation between $G$ and $H$.

By construction, $G = D_A + B - P_\infty$ and $H = D_A + \iota(B) - P_\infty$.  Assuming
 $u_A \mid 2v_A + h$
means that for every root $\xi$ of $u_A$ we have $2v_A(\xi) + h(\xi) = 0$.
At an affine point $P = (\xi, v_A(\xi))$ of $D_A$, the condition
$2y_P + h(x_P) = 0$ is exactly the ramification condition: $P = \iota(P)$.
So indeed every point of $D_A$ is a ramification point, and therefore $\iota(D_A) = D_A$. This gives us
\[
2D_A = \operatorname{div}(u_A) + 2gP_\infty,
\]
since $\deg u_A = g$.
On the other hand, since $B + \iota(B) = \operatorname{div}(u_B) + 2tP_\infty$, adding $G$ and $H$, we have
\begin{align*}
G + H &= 2D_A + B + \iota(B) - 2P_\infty \\
      &= \operatorname{div}(u_A) + 2gP_\infty + \operatorname{div}(u_B) + 2tP_\infty - 2P_\infty \\
      &= \operatorname{div}(u_A u_B) + (2g+2t-2)P_\infty.
\end{align*}
That is, $(2g+2t-2)P_\infty - G = H - \operatorname{div}(u_A u_B)$. Therefore, any function $\varphi \in \mathcal{L}((2g+2t-2)P_\infty - G)$ can be written as $\varphi = u_A u_B z$ for some $z \in \mathcal{L}(H)$. 

Substituting $\varphi$ into the evaluation of the dual code, we get
\begin{align*}
\CL(\mathcal D, G)^\perp & = \left\{ \left( \frac{u_A(x_P) u_B(x_P) z(P)}{p_\mathcal{D}'(x_P)(2y_P+h(x_P))} \right)_{P \in \Supp(\mathcal D)} \mathrel{\Bigg|} z \in \mathcal{L}(H) \right\} \\
& = \mathbf{m} \cdot \CL(\mathcal D,H).
\end{align*}

Lastly, we prove that $\mathcal{C} = \boldsymbol{\mu} \cdot \CL(\mathcal D, G)$ is an LCD code. The dual of $\mathcal{C}$ is given by
\[
\mathcal C^\perp = \boldsymbol{\mu}^{-1} \cdot  \CL(\mathcal D, G)^\perp = \boldsymbol{\mu}^{-1} \cdot \mathbf{m} \cdot  \CL(\mathcal D,H).
\]
Since we chose $\mu_P^2 = m_P$, we have $\mu_P^{-1} m_P = \mu_P$. Thus, $\mathcal{C}^\perp = \boldsymbol{\mu} \cdot \CL(\mathcal D,H)$. 
Hence,
\[
\mathcal{C} \cap \mathcal{C}^\perp = (\boldsymbol{\mu} \cdot \CL(\mathcal D, G)) \cap (\boldsymbol{\mu} \cdot \CL(\mathcal D, H)) = \boldsymbol{\mu} \cdot  (\CL(\mathcal D, G) \cap \CL(\mathcal D, H)).
\]
By Theorem~\ref{thm:LCP-hyperelliptic-Goppa}, $\CL(\mathcal D, G) \cap \CL(\mathcal D, H) = \{0\}$. Since $\boldsymbol{\mu}$ is a vector of non-zero elements, the intersection $\mathcal{C} \cap \mathcal{C}^\perp$ is exactly $\{0\}$, proving that $\mathcal{C}$ is an LCD code.
\end{proof}

\begin{proposition}
	\label{prop:lcd-y^2=x^q+x}
	Let $q$ be an odd prime power and let $\mathcal X : y^2 = x^q + x$ be the maximal hyperelliptic curve over $\mathbb{F}_{q^2}$. Let $k$ be a divisor of $q-1$, and let $S_k \le \mathbb{F}_q^*$ be the unique subgroup of order $k$. Define the evaluation divisor
	\[
	\mathcal D = \sum_{\substack{a \in \mathbb{F}_{q^2} \\ a^q + a \in S_k}} (P_a + \iota(P_a)),
	\]
	which has degree $2kq$.
	
	Let $(u_A,0)$ be a Mumford pair where $u_A \in \mathbb{F}_{q^2}[x]$ is a divisor of $x^q + x$ of degree $g = (q-1)/2$ and let $(u_B,v_B)$ be a Mumford pair where $u_B \in \mathbb{F}_{q^2}[x]$ is an irreducible polynomial of degree $kq$. 
	Assume that for every $a \in \mathbb{F}_{q^2}$ with $a^q+a \in S_k$ we have
	\[
	u_A(a)u_B(a) \neq 0
	\]
	and that the element
	\[
	m_a := \frac{(a^q+a)u_A(a)u_B(a)}{2\,y_a\,k} \in \mathbb{F}_{q^2}^*
	\]
	is a square in $\mathbb{F}_{q^2}^*$.
	Let $G$ be the divisor obtained from $(u_A,0)$ and $(u_B,v_B)$ as in Theorem~\ref{thm:LCP-hyperelliptic-Goppa}. Then, there exists a vector $\boldsymbol{\mu} \in (\mathbb{F}_{q^2}^*)^{2kq}$ such that $\mu_P^2 = m_P$ for every $P \in \operatorname{Supp}(\mathcal D)$ (where $m_P$ is the multiplier from Theorem~\ref{thm:duality-two-torsion-lcd}), and the code
	\[
	\mathcal C = \boldsymbol{\mu} \cdot \CL(\mathcal D, G)
	\]
	is an LCD code with parameters $[2kq,\,kq]_{q^2}$.
\end{proposition}

\begin{proof}
	We first determine the length of the code. Let $\operatorname{Tr} \colon \F_{q^2} \to \F_q$ be the trace map from $\F_{q^2}$ to $\F_q$. Since $S_k$ has order $k$ and the trace equation $\operatorname{Tr}(x)=x^q+x = c$ has exactly $q$ distinct roots in $\mathbb{F}_{q^2}$ for any $c \in \mathbb{F}_q$, the evaluation set contains exactly $kq$ affine $x$-coordinates. Since $0 \notin S_k$, these coordinates avoid the ramification points, lifting to $2kq$ distinct points on the curve. Thus, $\deg(\mathcal D) = 2kq$. The polynomial vanishing on these $x$-coordinates is $p_\mathcal{D}(x) = (x^q+x)^k - 1$ and its derivative $p_\mathcal{D}'(x)$ is given by
	\[
	p_\mathcal{D}'(x) = k(x^q+x)^{k-1}(qx^{q-1}+1)= k(x^q+x)^{k-1}.
	\]
	Observe that for any root $a$ of $p_\mathcal{D}(x)$, it holds $(a^q+a)^k = 1$; in particular, $(a^q+a)^{k-1} = \frac{1}{a^q+a}$. 
	Substituting this back yields $p_\mathcal{D}'(a) = \frac{k}{a^q+a}$. 
	
	We apply Theorem~\ref{thm:duality-two-torsion-lcd}. Since $h(x)=0$ and $v_A(x)=0$, we have $2v_A+h=0$; so the condition $u_A \mid (2v_A+h)$ is trivially satisfied. 
	Moreover,
	\[
	\CL(\mathcal D, G)^\perp = \mathbf{m} \cdot \CL(\mathcal D, H),
	\]
	where for each $P = (x_P, y_P) \in \operatorname{Supp}(\mathcal D)$,
	\[
	m_P = \frac{u_A(x_P)u_B(x_P)}{p_{\mathcal D}'(x_P)\,(2y_P)}.
	\]
	For $P_a = (a, y_a)$ we obtain
	\[
	m_{P_a} = \frac{u_A(a)u_B(a)}{\frac{k}{a^q+a} \cdot 2y_a} = \frac{(a^q+a)u_A(a)u_B(a)}{2y_a\,k} = m_a .
	\]
	For the conjugate point $\iota(P_a) = (a, -y_a)$,
	\[
	m_{\iota(P_a)} = -\,m_a .
	\]
	By hypothesis each $m_a$ is a square in $\mathbb{F}_{q^2}^*$. Because $q^2$ is a perfect square, $-1$ is also a square in $\mathbb{F}_{q^2}$.  Fix $\zeta \in \mathbb{F}_{q^2}^*$ with $\zeta^2 = -1$. Choose for each $a$ an element $\mu_a \in \mathbb{F}_{q^2}^*$ such that $\mu_a^2 = m_a$. Define
	\[
	\mu_{P_a} = \mu_a , \qquad \mu_{\iota(P_a)} = \zeta\,\mu_a .
	\]
	Then $\mu_P^2 = m_P$ for every $P \in \operatorname{Supp}(\mathcal D)$. By Theorem~\ref{thm:duality-two-torsion-lcd},  $\mathcal C = \boldsymbol{\mu} \cdot \CL(\mathcal D, G)$ is an LCD code over $\F_{q^2}$. Its length is $n=\deg \mathcal D = 2kq$ and its dimension, by Theorem~\ref{thm:LCP-hyperelliptic-Goppa}, equals $k_1=kq$. Hence $\mathcal C$ is an LCD $[2kq, kq]_{q^2}$-code.
\end{proof}

\begin{remark}
	\label{rem:MDS-k=1}
	For $k=1$, the resulting LCD code has length $2q$ and dimension $q$. By Proposition~\ref{prop:practical-theta-avoidance}, achieving the MDS property requires the class $\beta=[B-qP_\infty]$ to simultaneously yield $u$-polynomials of degree $g$ for all $2^q$ subsets of $\operatorname{Supp}(\mathcal D)$ of size $q$. This simultaneous condition is genuinely nontrivial; merely comparing the $2^q$ subsets to $|J(\mathcal X)| = (q+1)^{q-1}$ does not formally guarantee existence. Nevertheless, the explicit examples below indicate that such admissible Mumford data generally exists, motivating the following expectation.
\end{remark}
\begin{conjecture}
\label{exp:MDS-LCD-q}
    For every odd prime power $q \ge 5$, there exists an irreducible polynomial $u_B\in\mathbb F_{q^2}[x]$ of degree $q$ satisfying both the square condition of Proposition~\ref{prop:lcd-y^2=x^q+x} and the MDS criterion of Proposition~\ref{prop:practical-theta-avoidance}, so that the associated LCD code $\mathcal C = \boldsymbol{\mu} \cdot \CL(\mathcal D,G)$ has parameters $[2q, q, q+1]_{q^2}$.
\end{conjecture}

A full verification of this expectation would require controlling simultaneously the square condition and the avoidance of all $2^q$ forbidden Jacobian classes. We leave this as an open problem. The examples below illustrate three different aspects of the construction. Example~\ref{ex:ex-2} gives an MDS LCP in characteristic two. Example~\ref{ex:ex-3} verifies the full MDS--LCD construction for the odd prime power $q=5$ and therefore provides direct computational support for Expectation~\ref{exp:MDS-LCD-q}. Example~\ref{ex:ex-4} gives a further LCD example for $q=7$ and $k=2$.
	
	All computations and algorithmic verifications in the following examples were performed using the computer algebra system Magma \cite{magma}.	
    
\begin{example}
\label{ex:ex-2}
Let $q=4$ and $\F_{16} = \F_2(\zeta)$ with $\zeta^4+\zeta+1=0$. Let 
\[
		\mathcal X : y^2 + y = x^5
		\]
		be a hyperelliptic curve over $\F_{16}$ of genus $g = 2$, having a single rational point $P_\infty$ at infinity. The curve has $33$ rational points over $\F_{16}$, and its hyperelliptic involution is given by $\iota(x,y) = (x, y+1)$ for affine points and $\iota(P_\infty)=P_\infty$. 
		
		Define the evaluation polynomial 
		\[
		p_\mathcal{D}(x)=(x + \zeta^2)(x + \zeta^4)(x + \zeta^9)(x + \zeta^{13}),
		\]
		and choose the Mumford pairs 
		\[
		(u_A(x), v_A(x)) = \left(x^2 + \zeta^6 x + \zeta^{10}, \ \zeta x + \zeta^7 \right)
		\]
		and
		\[
		\left( u_B(x), v_B(x) \right)  = \left( x^4 + \zeta^5x^3 + \zeta^7x^2 + \zeta^9x + \zeta^{13}, \  x^3 + x^2 + \zeta^{10}x + \zeta^4 \right).
		\]
		The polynomial $u_A$ is irreducible of degree $2$ over $\F_{16}$. Moreover, $u_B(x) = (x + \zeta^{10})(x + \zeta^5)(x + \zeta^6)(x + \zeta^7)$ splits completely over $\F_{16}$. With our choice of $v_B(x)$, the affine divisor $D_{(u_B, v_B)}$ is a sum of four $\F_{16}$-rational points of $\X$; namely,
		\[
		D_{(u_B, v_B)} = (\zeta^{10}, \zeta) + (\zeta^5, \zeta^2) + (\zeta^6, \zeta^5) + (\zeta^7, \zeta).
		\]
		Clearly, $\gcd(u_A, u_B) = 1$. By the Chinese Remainder Theorem, there exist unique polynomials $v_G, v_H \in \F_{16}[x]$ modulo $u_A u_B$ satisfying the congruences
		\begin{align*}
			v_G &\equiv \zeta x + \zeta^7 \pmod{u_A(x)}, \\
			v_G &\equiv x^3 + x^2 + \zeta^{10}x + \zeta^4 \pmod{u_B(x)}
		\end{align*}
		and, since $v_B(x) + 1 = x^3 + x^2 + \zeta^{10}x + (\zeta^4 + 1) = x^3 + x^2 + \zeta^{10}x + \zeta$,
		\begin{align*}
			v_H &\equiv \zeta x + \zeta^7 \pmod{u_A(x)}, \\
			v_H &\equiv x^3 + x^2 + \zeta^{10}x + \zeta \pmod{u_B(x)}.
		\end{align*}
		Solving this system of CRT equations yields 
		\[
		v_G(x) = \zeta^5x^5 + \zeta^6 x^4 + x^3 + x^2 + \zeta^{13} x + \zeta^8
		\]
		and 
		\[
		v_H(x) = x^5 + \zeta^6 x^4 + \zeta^4 x^3 + \zeta^{14} x^2 + \zeta x + \zeta^{14}.
		\]
		Let $G$ and $H$ be the divisors induced by the pairs $(u_A u_B, v_G)$ and $(u_A u_B, v_H)$, respectively. The roots of the polynomial $p_{\mathcal{D}}(x)$ define an evaluation divisor $\mathcal{D}$ supported on $2t = 8$ distinct rational points, forming $t = 4$ conjugate pairs. Furthermore, $\Supp(\mathcal{D})$ is disjoint from $\Supp(G) \cup \Supp(H)$. We construct the associated AG codes 
		\[
		\C_1 = \CL(\mathcal D, G) \quad \text{and} \quad \C_2 = \CL(\mathcal D, H).
		\]
		According to Theorem \ref{thm:LCP-hyperelliptic-Goppa}, the pair $(\C_1, \C_2)$ forms an LCP of codes over $\F_{16}$. By Theorem \ref{thm:MDS-LCP-Mumford-data}, $\C_1$ and $\C_2$ are guaranteed to be MDS codes if, for all $\binom{8}{4} = 70$ subsets $I \subseteq \operatorname{Supp}(\mathcal{D})$ of size $4$, the reduced Mumford representatives of the associated divisor classes in the Jacobian have $u$-polynomials of degree exactly $g=2$.
		
		An algorithmic verification of these $70$ conditions confirms that all such $u$-polynomials indeed have the required degree. Hence, the hypothesis of Theorem \ref{thm:MDS-LCP-Mumford-data} is fully satisfied, and the pair $(\C_1, \C_2)$ forms an LCP of MDS hyperelliptic codes over $\F_{16}$ with parameters $[8, 4, 5]_{16}$.
	\end{example}
	
	\begin{example}
    \label{ex:ex-3}
		Let $q=5$ and $\mathbb{F}_{25}=\mathbb{F}_{5}(\zeta)$ with $\zeta^{2}-\zeta+2=0$.
		Consider the hyperelliptic curve
		\[
		\mathcal{X}:y^{2}=x^{5}+x
		\]
		over $\mathbb{F}_{25}$, of genus $g=2$. We select the Mumford pairs
		\[
		(u_{A}(x),\, v_{A}(x)) = (x^{2}+2,\, 0)
		\]
		and $(u_B(x), v_B(x))$ given by
		\begin{align*}
			u_{B}(x) &= x^{5}+(3\zeta+3)x^{4}  + (2\zeta+2)x^{3} +(2\zeta+1)x^{2} +  (2\zeta+1)x+4,  \\
			v_{B}(x) &= 3\zeta x^{3}+(2\zeta+4)x^{2}+(\zeta+4)x+3\zeta.
		\end{align*}		
		Let $G$ be the divisor obtained from $(u_{A},v_{A})$ and $(u_{B},v_{B})$
		as in Theorem~\ref{thm:LCP-hyperelliptic-Goppa}. Specifically, its corresponding polynomial data is
		\begin{align*}
			u_{G}(x) &= x^{7}+(3\zeta+3)x^{6}+(2\zeta+4)x^{5}+(3\zeta+2)x^{4}+\zeta x^{3}+(4\zeta+1)x^{2}+(4\zeta+2)x+3, \\
			v_G(x) &= (4\zeta+4)x^{6}+x^{5}+(2\zeta+2)x^{4}+3x^{3}+4x^{2}+2x+4\zeta+2.
		\end{align*}		
		Setting $k=1$, the pairs $(u_{A},v_{A})$ and $(u_{B},v_{B})$ satisfy the conditions of Proposition~\ref{prop:lcd-y^2=x^q+x}. Consequently, there exists a vector $\boldsymbol{\mu}\in\left(\mathbb{F}_{25}^{*}\right)^{10}$ such that $\boldsymbol{\mu}\cdot\mathcal{C}_{\mathcal{L}}(\mathcal{D},G)$ is an LCD code, where $\mathcal{D}$ denotes the sum of the $10$ rational places whose $x$-coordinates are roots of $x^5+x-1$. An algorithmic verification of the associated $\binom{10}{5} = 252$ conditions confirms that all reduced Mumford representatives possess a $u$-polynomial of degree $g = 2$. Hence, the resulting LCD code is MDS with parameters $[10, 5, 6]_{25}$.
	\end{example}

	\begin{example}
    \label{ex:ex-4}
		Let $q=7$ and $\mathbb{F}_{49}=\mathbb{F}_{7}(\zeta)$ with $\zeta^{2}-\zeta+3=0$.
		Define the hyperelliptic curve
		\[
		\mathcal{X}:y^{2}=x^{7}+x
		\]
		over $\mathbb{F}_{49}$, of genus $g=3$.  We choose the Mumford pairs
		\[
		(u_{A}(x), \, v_{A}(x)) = (x^{3}+x, \, 0 )
		\]
		and $(u_{B}(x), v_{B}(x))$, where
		\begin{align*}
			u_{B}(x) &= x^{14} +(6\zeta+5)x^{13}+(2\zeta+1)x^{12}+(3\zeta+6)x^{11}+(5\zeta+6)x^{10} \\
			&\quad +(2\zeta+4)x^{9}+(5\zeta+4)x^{8}+(3\zeta+2)x^{7}+(6\zeta+2)x^{6} \\
			&\quad +(4\zeta+2)x^{5}+(\zeta+4)x^{4}+(6\zeta+6)x^{3}+6\zeta x^{2} +(6\zeta+4)x+4\zeta+5, \\
			v_{B}(x) &= (4\zeta+2)x^{13}+(3\zeta+2)x^{12}+(2\zeta+5)x^{11}+(\zeta+1)x^{10} +(4\zeta+2)x^{9}  \\
			&\quad +(4\zeta+6)x^{8} +(4\zeta+4)x^{7}+(4\zeta+2)x^{6}+2\zeta x^{5}+(6\zeta+4)x^{4} \\  
			&\quad +(2\zeta+3)x^{3}+6\zeta x^{2}+(2\zeta+3)x+6\zeta+6.
		\end{align*}
		Let $G$ be the divisor obtained from $(u_{A},v_{A})$ and $(u_{B},v_{B})$
		as in Theorem~\ref{thm:LCP-hyperelliptic-Goppa}.  		
		Setting $k=2$, the pairs $(u_{A},v_{A})$ and $(u_{B},v_{B})$ satisfy the conditions of Proposition~\ref{prop:lcd-y^2=x^q+x}. Therefore, there exists a vector $\boldsymbol{\mu}\in\left(\mathbb{F}_{49}^{*}\right)^{28}$ such that $\boldsymbol{\mu}\cdot\mathcal{C}_{\mathcal{L}}(\mathcal{D},G)$ is an LCD code with parameters $[28,14]_{49}$, where $\mathcal{D}$ denotes the sum of the $28$ rational places whose $x$-coordinates are roots of $(x^7 + x)^2 - 1$.
	\end{example}

	\section*{Acknowledgements}
We thank Akiyoshi Sannai for carefully reading an earlier version of the manuscript, for pointing out an error in the duality multiplier in Theorem~\ref{thm:duality-two-torsion-lcd} and its consequence for Proposition~\ref{prop:lcd-y^2=x^q+x}, and for providing a counterexample that helped us correct the statement and proof. 

Saeed Tafazolian was partially supported by CNPq grant no.~302774/2025-4 and FAPESP grant no.~2024/00923-6. Yuri da Silva was supported by FAPESP grant no.~2026/09571-0.


\begin{thebibliography}{99}
		
	\bibitem{Ballet2006}
		S.~Ballet and D.~Le~Brigand,
		``On the existence of non-special divisors of degree $g$ and $g-1$ in
		algebraic function fields over $\mathbb{F}_q$,''
		\emph{Journal of Number Theory}, vol.~116, no.~2, pp.~293--310, 2006.
		
		\bibitem{BK2025}
		S.~Ballet and M.~Koutchoukali,
		``On the non-special divisors in algebraic function fields defined over
		finite fields,''
		\emph{Polynesian Journal of Mathematics}, vol.~2, no.~1, pp.~1--41, 2025.
		
		\bibitem{Beelen2018}
		P.~Beelen and L.~Jin, Explicit MDS codes with complementary duals.
		\emph{IEEE Trans. Inform. Theory} 64 (2018), 7188--7193.
		
		\bibitem{BeelenTRS2017}
		P.~Beelen, S.~Puchinger, and J.~Rosenkild\'{e},
		``Twisted Reed--Solomon codes,''
		in \emph{Proc. IEEE Int. Symp. Inf. Theory (ISIT)}, 2017, pp.~334--336.
		
		\bibitem{BeelenStructural2018}
		P.~Beelen, M.~Bossert, S.~Puchinger, and J.~Rosenkild\'{e},
		``Structural properties of twisted Reed--Solomon codes with applications to
		code-based cryptography,''
		in \emph{Proc. IEEE Int. Symp. Inf. Theory (ISIT)}, 2018, pp.~946--950.
		
		\bibitem{Bhowmick2024}
		S.~Bhowmick, D.~K.~Dalai, and S.~Mesnager,
		``On linear complementary pairs of algebraic geometry codes over finite
		fields,''
		\emph{Discrete Math.}, vol.~347, no.~12, p.~114193, 2024.
		
		\bibitem{magma}
		W. Bosma, J. Cannon, and C. Playoust, 
		\emph{The Magma algebra system. I. The user language}, 
		J. Symbolic Comput. \textbf{24} (1997), no. 3-4, 235--265. 
		Computational algebra and number theory (London, 1993).
		
		\bibitem{Bringer14}
		J.~Bringer, C.~Carlet, H.~Chabanne, S.~Guilley, and H.~Maghrebi,
		Orthogonal direct sum masking: A smartcard friendly computation paradigm
		in a code, with builtin protection against side-channel and fault attacks,
		in \emph{Proc. IFIP Int. Workshop Inf. Secur. Theory Pract.},
		Springer, 2014, pp.~40--56.
		
		\bibitem{Cantor1987}
		D.~G.~Cantor,
		Computing in the Jacobian of a hyperelliptic curve.
		\emph{Mathematics of Computation}, 48(177):95--101, 1987.
		
		\bibitem{Carlet2018}
		C.~Carlet, C.~G\"{u}neri, F.~\"{O}zbudak, B.~\"{O}zkaya, and P.~Sol\'{e},
		``On linear complementary pairs of codes,''
		\emph{IEEE Trans. Inf. Theory}, vol.~64, no.~10, pp.~6583--6589, 2018.
		
		\bibitem{CarletGuilley2016}
		C.~Carlet and S.~Guilley,
		``Complementary dual codes for counter-measures to side-channel attacks,''
		\emph{J. Adv. Math. Commun.}, vol.~10, no.~1, pp.~131--150, 2016.
		
		\bibitem{CarletEquiv2018}
		C.~Carlet, S.~Mesnager, C.~Tang, Y.~Qi, and R.~Pellikaan,
		``Linear codes over $\mathbb{F}_q$ are equivalent to LCD codes for $q > 3$,''
		\emph{IEEE Trans. Inf. Theory}, vol.~64, no.~4, pp.~3010--3017, 2018.
		
		\bibitem{CarletMDS2018}
		C.~Carlet, S.~Mesnager, C.~Tang, and Y.~Qi,
		``Euclidean and Hermitian LCD MDS codes,''
		\emph{Des. Codes Cryptogr.}, vol.~86, no.~11, pp.~2605--2618, 2018.
		
		\bibitem{Castellanos2025}
		A.~S.~Castellanos, A.~V.~Marques, and L.~Quoos,
		``Linear complementary dual codes and linear complementary pairs of AG
		codes in function fields,''
		\emph{IEEE Trans. Inf. Theory}, vol.~71, no.~3, pp.~1676--1688, 2025.
		
		\bibitem{ChenLiu2018}
		B.~Chen and H.~Liu,
		``New constructions of MDS codes with complementary duals,''
		\emph{IEEE Trans. Inf. Theory}, vol.~64, no.~8, pp.~5776--5782, 2018.
		
		\bibitem{MBoerThesis}
		M.~A.~de~Boer,
		\emph{Codes: Their Parameters and Geometry},
		Ph.D.\ thesis, Technische Universiteit Eindhoven,
		Mathematics and Computer Science, 1997.
		
		\bibitem{Goppa1981}
		V.~D.~Goppa,
		Codes on algebraic curves.
		In \emph{Sov. Math.-Dokl}, volume~24, pages 170--172, 1981.
		
		\bibitem{4author}
		J.~Huang, H.~Chen, H.~Zhang, and C.~A.~Zhao,
		``Linear complementary pairs of algebraic geometry codes via Kummer
		extensions,''
		\emph{arXiv preprint arXiv:2506.23081}, 2025.
		
		\bibitem{Jin2017}
		L.~Jin,
		``Construction of MDS codes with complementary duals,''
		\emph{IEEE Trans. Inf. Theory}, vol.~63, no.~5, pp.~2843--2847, 2017.
		
		
		\bibitem{LiuLiu2021}
		H.~Liu and S.~Liu,
		``Construction of MDS twisted Reed--Solomon codes and LCD MDS codes,''
		\emph{Des. Codes Cryptogr.}, vol.~89, pp.~2051--2065, 2021.
		
		\bibitem{MQ2025}
		A. Marques and L. Quoos, \emph{Sequences of LCD AG codes and LCP of AG Codes attaining the Tsfasman-Vladut-Zink bound}, arXiv preprint arXiv:2505.23937 [math.AG], 2025.
		
		\bibitem{MST2026}
		A.~Marques, Y.~da~Silva, and S.~Tafazolian,
		Construction of Non-Special Divisors on Kummer Covers with Arbitrary
		Ramification for LCP Codes,
		\emph{arXiv preprint arXiv:2605.14046}, 2026.
		
		\bibitem{Massey1992}
		J.~L.~Massey,
		``Linear codes with complementary duals,''
		\emph{Discrete Math.}, vol.~106--107, pp.~337--342, 1992.
		
		\bibitem{Mendoza2026}
		E.~Mendoza, H.~Navarro, and L.~Quoos,
		``Characterization of non-special divisors of small degree on Kummer
		extensions and LCP codes,''
		\emph{arXiv preprint arXiv:2604.27146}, 2026.
		
		\bibitem{Moreno2024}
		E.~C.~Moreno, H.~H.~Lopez, and G.~L.~Matthews,
		``Explicit non-special divisors of small degree, algebraic geometric
		hulls, and LCD codes from Kummer extensions,''
		\emph{SIAM J. Appl. Algebra Geometry}, vol.~8, no.~2, pp.~394--413, 2024.
		
		\bibitem{Mum84}
		D.~Mumford,
		\emph{Tata Lectures on Theta {II}: Jacobian Theta Functions and
			Differential Equations},
		vol.~43, Birkh\"{a}user Boston, 1984.
		
		\bibitem{QianZhang2018}
		J.~Qian and L.~Zhang,
		``On MDS linear complementary dual codes and entanglement-assisted quantum
		codes,''
		\emph{Des. Codes Cryptogr.}, vol.~86, no.~7, pp.~1565--1572, 2018.
		
		\bibitem{Sheekey2016}
		J.~Sheekey,
		``A new family of linear maximum rank distance codes,''
		\emph{Adv. Math. Commun.}, vol.~10, no.~3, pp.~475--488, 2016.
		
		\bibitem{ShiYueYang2018}
		X.~Shi, Q.~Yue, and S.~Yang,
		``New LCD MDS codes constructed from generalized Reed--Solomon codes,''
		\emph{J. Algorithm Appl.}, vol.~1, p.~1950150, 2018.
		
		\bibitem{Stichtenoth2009}
		H.~Stichtenoth,
		\emph{Algebraic Function Fields and Codes},
		Graduate Texts in Mathematics, vol.~254, 2nd~ed.,
		Springer-Verlag, Berlin, 2009.
		
		\bibitem{TVZ1982}
		M.~A. Tsfasman, S.~G. Vl\u{a}du\c{t}, and T.~Zink.
		\newblock Modular curves, {S}himura curves, and {G}oppa codes, better than {V}arshamov-{G}ilbert bound.
		\newblock {\em Mathematische Nachrichten}, 109(1):21--28, 1982.
\end{thebibliography}
\end{document}